\documentclass[11pt]{amsart}%{article}
\usepackage[T1]{fontenc}
\usepackage{lmodern,microtype,amsmath,amssymb,amsthm,mathtools}
\usepackage[margin=1in]{geometry}
\usepackage{xcolor,cite}
\usepackage[hypertexnames=false,colorlinks=true,linkcolor=blue!45!black,citecolor=blue!45!black,urlcolor=blue!45!black]{hyperref}
\hypersetup{
  pdftitle={Convergence rates in the periodic homogenization of vanishing-viscous Hamilton-Jacobi equations},
  pdfauthor={***}
}
\allowdisplaybreaks[1]
\newtheorem{theorem}{Theorem}[section]
\newtheorem{proposition}[theorem]{Proposition}
\newtheorem{lemma}[theorem]{Lemma}
\theoremstyle{remark}

\newcommand{\R}{\mathbb R}\newcommand{\Z}{\mathbb Z}
\newcommand{\T}{\mathbb T}\newcommand{\Hb}{\overline H}
\newcommand{\Lb}{\overline L}
\newcommand{\dd}{\,\mathrm d}\newcommand{\Pe}{\mathcal P_\varepsilon}
\newcommand{\HYP}{\textup{(H)}}\newcommand{\GYP}{\textup{(G)}}
\numberwithin{equation}{section}
\title[]{Convergence rates in the periodic homogenization of vanishing--viscous Hamilton--Jacobi equations}

\author{Kai Qin}
\address[K. Qin]{Qiuzhen College, Tsinghua University, Beijing 100084}
\email{qin-j21@mails.tsinghua.edu.cn}

\date{September 2026}

\begin{document}
\begin{abstract}
We study convergence rates when periodic homogenization and vanishing
viscosity occur at the same scale in Hamilton--Jacobi equations whose
momentum Hessians are positive definite at every point.
For bounded Lipschitz initial data and a class of smooth Hamiltonians,
we establish an optimal rate
$O(\varepsilon|\log\varepsilon|)$ on fixed time intervals.
The second-order term plays a key role: the elliptic cell problem
provides higher derivative bounds for the effective Hamiltonian
$\Hb$ and its Legendre dual $\Lb$.
The proof is purely PDE, based on corrector expansions and viscosity
comparison, while its guiding idea comes from the control interpretation
of the Hopf--Lax formula. For a fixed target $(x,t)$, a minimizing origin selects a characteristic
velocity and its dual momentum, these guide the construction of smooth comparison profiles, whose gradients supply the momentum argument of the first-order corrector.
Thus we never need to differentiate the effective solution, even at points where it is nonsmooth.
\end{abstract}
\maketitle


\tableofcontents
\clearpage

\section{Introduction}\label{sec:intro}
\subsection{Setting and main idea}\label{subsec:setting}
Fix an integer $n\ge1$ and let $\T^n=\R^n/\Z^n$.
Write $B_r=\{p\in\R^n:|p|<r\}$ and
$\overline B_r=\{p\in\R^n:|p|\le r\}$.
The variable $y\in\T^n$ is periodic position, and $p\in\R^n$
is the gradient argument. We write $H^*$ for the original
Hamiltonian and $H$ for its fixed modification. We assume
\begin{equation}\tag{H}\label{hyp:H}
\begin{gathered}
 H^*\in C^3(\T^n\times\R^n),\qquad H^*_{pp}(y,p)>0,\\
 \lim_{r\to\infty}\inf_{\substack{y\in\T^n\\|p|\ge r}}
 \bigl\{H^*(y,p)^2+nH^*_y(y,p)\cdot p\bigr\}=+\infty.
\end{gathered}
\end{equation}
Here $H^*_y=D_yH^*$ and $H^*_{pp}=D^2_{pp}H^*$. For a fixed $K\ge0$, the initial datum satisfies
\begin{equation}\tag{G}\label{hyp:G}
 g\in L^\infty(\R^n),\qquad
 |g(z)-g(z')|\le K|z-z'|\quad(z,z'\in\R^n).
\end{equation}
The effective Hamiltonians of $H^*$ and $H$ are denoted by
$\Hb^*$ and $\Hb$, respectively; their cell problems are specified
at the beginning of Section~\ref{sec:prelim}.
Our first step replaces $H^*$ by a uniformly convex Hamiltonian,
with an error of order $\varepsilon$.

\begin{theorem}\label{thm:reduction}
Assume \HYP--\GYP. There exist $R=R(H^*,n,K)>K+1$ and
$H\in C^3(\T^n\times\R^n)$, both independent of $\varepsilon$, such that
\begin{equation}\label{eq:reduction-H}
 H=H^*\ \text{on }\T^n\times\overline B_R,\qquad
 0<\lambda I\le H_{pp}\le\Lambda I,\qquad |H_y|\le C_H.
\end{equation}
Here $I$ is the identity matrix, and $\lambda,\Lambda,C_H>0$
depend only on $H^*,n,K$. For every $0<\varepsilon\le1$, the following
four Cauchy problems on $\R^n\times(0,\infty)$ have unique viscosity
solutions that are bounded on $\R^n\times[0,T]$ for every $T>0$:
\begin{subequations}\label{eq:four-equations}
\begin{align}
 U_t^\varepsilon+H^*(x/\varepsilon,DU^\varepsilon)
 &=\tfrac\varepsilon2\Delta U^\varepsilon,
 &U^\varepsilon(\cdot,0)&=g,\label{eq:pde-original}\\
 U_t+\Hb^*(DU)&=0,&U(\cdot,0)&=g,\label{eq:effective-original}\\
 u_t^\varepsilon+H(x/\varepsilon,Du^\varepsilon)
 &=\tfrac\varepsilon2\Delta u^\varepsilon,
 &u^\varepsilon(\cdot,0)&=g,\label{eq:reduced-equations}\\
 u_t+\Hb(Du)&=0,&u(\cdot,0)&=g.\label{eq:modified-effective}
\end{align}
\end{subequations}
Moreover,
\begin{equation}\label{eq:reduction-error}
 U=u,\qquad \operatorname{Lip}_x u(\cdot,t)\le K\quad(t\ge0),\qquad
 \|U^\varepsilon-u^\varepsilon\|_{L^\infty(\R^n\times[0,\infty))}
 \le2K\varepsilon.
\end{equation}
\end{theorem}

The second-order term makes the cell problem uniformly elliptic.
Under the following assumptions, this gives $C^3$ regularity of the
effective Hamiltonian and its convex dual.
\begin{theorem}\label{thm:cell}
Suppose that, for some $\lambda,\Lambda,C_H>0$,
\begin{equation}\label{eq:uniform-structure}
 H\in C^3(\T^n\times\R^n),\qquad
 0<\lambda I\le H_{pp}\le\Lambda I,\qquad
 |H_y(y,p)|\le C_H(1+|p|^2).
\end{equation}
For every $p\in\R^n$, the cell problem \eqref{eq:cell} has a unique
normalized classical pair $(\chi(\cdot,p),\Hb(p))$.
Let $\Lb(q)=\sup_{p\in\R^n}\{p\cdot q-\Hb(p)\}$ be its convex dual. Then
\[
 \chi\in C^2(\T^n\times\R^n),\qquad \Hb,\Lb\in C^3(\R^n),
 \qquad D^2\Hb>0,\quad D^2\Lb>0.
\]
In particular, this applies to the Hamiltonian in
Theorem~\ref{thm:reduction}.
\end{theorem}

Our main result is the following global rate.
\begin{theorem}\label{thm:global}
Under \HYP--\GYP, there is $C=C(H^*,n,K)$ such that
\begin{equation}\label{eq:global}
 |U^\varepsilon(x,t)-U(x,t)|
 \le C\varepsilon\bigl[1+\log(1+t/\varepsilon)\bigr]
 \quad(x\in\R^n,\ t\ge0,\ 0<\varepsilon\le1).
\end{equation}
\end{theorem}

The constant $C(H^*,n,K)$ is independent of $\varepsilon,x,t$ and
of the datum $g$ beyond its Lipschitz bound $K$.
The rate proof has constants $C(H,n,K)$. Since $H$ is fixed by
$H^*,n,K$ in Theorem~\ref{thm:reduction}, they give the stated
constant $C(H^*,n,K)$.

For the main idea, fix a target $(x,t)$ with $t>0$ and a minimizing
point $z_*$ in the Hopf--Lax formula for $u$.
Consider the two $C^3$ profiles
\begin{align*}
 v_1(z,s)&=u(x,t)-(t-s)\Lb\!\left(\frac{x-z}{t-s}\right),
 &&0\le s<t,\\
 v_2(z,s)&=g(z_*)+s\Lb\!\left(\frac{z-z_*}{s}\right),
 &&s>0.
\end{align*}
The Hopf--Lax formula gives
\[
 v_1(z,s)\le u(z,s)\le v_2(z,s)\quad(0<s<t),\qquad
 \lim_{s\uparrow t}v_1(x,s)=u(x,t)=v_2(x,t).
\]
The minimizing point $z_*$ is the origin of an effective
characteristic reaching $(x,t)$. Its velocity is $(x-z_*)/t$, and
the upper profile supplies the corresponding momentum
$D_zv_2(x,t)=D\Lb((x-z_*)/t)$.
Thus the corrector can use smooth profile gradients even where
$Du(x,t)$ does not exist.
Away from their singular times, both profiles solve
$\partial_s v_i+\Hb(D_zv_i)=0$.
When the required derivatives are continuous and bounded, adding
the cell corrector and expanding the equation gives the homogenization
estimate. This suggests comparing
\[
 v_i(z,s)+\varepsilon\chi(z/\varepsilon,D_zv_i(z,s)),\qquad i=1,2,
\]
with $u^\varepsilon$ from below and above.

To handle the singularities, we add $\delta>0$ to the time factors
$t-s$ and $s$. We also restrict the momentum in the lower profile
by replacing $\Lb$ with a $C^3$ function $\Phi$; the upper
comparison is localized to a region of bounded gradient.
Small shifts enforce the initial inequalities. The corrected
profiles then become a subsolution and a supersolution after
integrating their equation errors. The leading integrals are
\[
 \varepsilon\int_0^t\frac{\dd s}{t+\delta-s}
 =\varepsilon\int_0^t\frac{\dd s}{s+\delta}
 =\varepsilon\log(1+t/\delta).
\]
Comparison, followed by $\delta=\varepsilon$ and
\eqref{eq:reduction-error}, gives the rate. Section~\ref{sec:prelim}
proves the reduction and regularity theorems. Section~\ref{sec:proof}
prepares the comparison estimates and proves the global rate.

\subsection{History of the problem}\label{subsec:history}
For first-order periodic homogenization, the qualitative theory goes back
to Lions, Papanicolaou, and Varadhan \cite{LPV} and the perturbed test
function method of Evans \cite{Evans89,Evans92}.
Capuzzo-Dolcetta and Ishii \cite{CDI} obtained the classical
$O(\varepsilon^{1/3})$ convergence rate without a convexity assumption.
In the convex setting, Mitake, Tran, and Yu \cite{MTY} proved an
$O(\varepsilon)$ lower error bound in every dimension and sharper
two-sided estimates in several special cases.
Tran and Yu \cite{TY} proved the optimal $O(\varepsilon)$ rate,
uniformly in space and time, by adapting Burago's curve-cutting argument
\cite{Burago} for periodic metrics to the time-dependent action.
Their method has proved flexible for first-order periodic homogenization
of convex Hamiltonians. Han, Jing, Mitake, and Tran \cite{HJMT}
extended it to state-constraint problems on periodically perforated
domains, obtaining the optimal $O(\varepsilon)$ rate.
Han and Jang \cite{HJ} treated multiscale problems with slow spatial
dependence, obtaining an $O(\sqrt\varepsilon)$ rate on fixed time intervals.
Recently, Guo, Jing, Tran, and Zhang \cite{GJTZ} obtained probabilistic
convergence estimates with exponent $1/2$, up to slowly varying factors,
for large-time averages and homogenization in dynamic random environments
with finite-range dependence in time.

The study of vanishing viscosity without rapid spatial oscillations
goes back to Fleming \cite{Fleming64}. The viscosity-solution theory
of Crandall and Lions \cite{CL83} and their approximation estimates
\cite{CL84} provide a comparison-based PDE framework, including the
classical $O(\sqrt\varepsilon)$ estimate for Lipschitz data. The nonlinear adjoint method gives another approach;
see Evans \cite{Evans10} and Tran \cite{Tran11}.
More recently, Chaintron and Daudin established the sharp
$O(\varepsilon|\log\varepsilon|)$ rate for purely quadratic Hamiltonians
with Lipschitz data and no forcing \cite{CDq}, and for uniformly convex
Hamiltonians with semiconcave data
under their regularity assumptions \cite{CD}.
Cirant and Goffi \cite{CG} obtained $O(\varepsilon|\log\varepsilon|)$
estimates away from the terminal time for uniformly convex equations on
the torus, as well as $O(\varepsilon^\beta)$ bounds for suitable
$\beta>1/2$ in a class of degenerately convex Hamiltonians.
Except for explicitly global
statements, the rates recalled here concern fixed finite time intervals.

When periodic homogenization and vanishing viscosity occur together,
Evans \cite{Evans92} established qualitative homogenization for viscous
Hamilton--Jacobi equations. Camilli, Cesaroni, and Marchi \cite{CCM}
studied quantitative estimates for the joint limit in fully nonlinear
elliptic equations. For the evolutionary problem at the scale considered
here, Qian, Sprekeler, Tran, and Yu \cite{QSTY} proved the
$O(\sqrt\varepsilon)$ rate by PDE methods without assuming convexity,
and showed its optimality in that general class.
For the quadratic Hamiltonian $H(y,p)=|p|^2/2+V(y)$,
Liu, Tran, and Yu \cite{LTY} obtained the sharp global bound
$C\varepsilon[1+\log(1+t/\varepsilon)]$ for bounded Lipschitz data,
using the Hopf--Cole transformation and estimates for the resulting
parabolic kernel. The present paper obtains this global logarithmic
bound under \HYP--\GYP\ by corrector expansions and viscosity
comparison, using the regularity furnished by the elliptic cell problem.

\subsection{What is sufficient and what remains}\label{subsec:scope}
The hypothesis on $H^*$ is the pointwise positive definiteness of
$H^*_{pp}$ specified after \eqref{hyp:H}. On each fixed momentum ball,
compactness gives a positive lower Hessian bound. The reduction in
Theorem~\ref{thm:reduction} then gives a uniformly convex modification.
This is enough to obtain positive definiteness of $D^2\Hb$ and the
local derivative bounds for $\Lb$ used in the comparison argument.

Strict convexity alone is weaker. For instance, $H^*(p)=|p|^4$ is
smooth and strictly convex, but $H^*_{pp}(0)=0$. In this nonoscillatory
case $\Hb^*=H^*$ and its convex dual is
\[
 \Lb^*(q)=\frac{3}{4^{4/3}}|q|^{4/3},
\]
which is not twice differentiable at the origin. Thus strict convexity
does not by itself provide the dual regularity needed here.
For merely convex Hamiltonians with an affine segment,
\cite[Theorem~1.2]{QSTY} gives one-dimensional examples with bounded
Lipschitz data for which $|u^\varepsilon(0,1)-u(0,1)|\ge c\sqrt\varepsilon$,
so the general $O(\sqrt\varepsilon)$ rate is sharp even in this class.
Even in the uniformly convex class, examples with bounded Lipschitz
data attaining order $\varepsilon|\log\varepsilon|$ exclude a general
$O(\varepsilon)$ bound \cite[Remark~3.7]{LTY}.

The $C^3$ assumption has a different role. In
Theorem~\ref{thm:cell}, it gives $C^3$ parameter dependence in the
cell problem and hence $\Hb,\Lb\in C^3$. Their third derivatives
are bounded on the compact sets used for the comparison profiles.
In Lemma~\ref{lem:calculation}, diffusion of the corrector produces
$-\varepsilon^2\chi_p\cdot D\Delta v/2$, which requires third
spatial derivatives of the smooth profile $v$.
With only $C^2$ regularity, mollified third derivatives need not remain
bounded, so the same proof does not settle the logarithmic rate.
This limitation concerns general Hamiltonians: the quadratic-potential
result of \cite{LTY} already allows a merely Lipschitz potential.

\section{Preliminary results}\label{sec:prelim}
Integration over $\T^n$ means integration over $[0,1]^n$.
For each $p\in\R^n$, the normalized cell problem for $H$ is
\begin{equation}\label{eq:cell}
 -\tfrac12\Delta_y\chi(y,p)+H(y,p+D_y\chi(y,p))=\Hb(p),
 \qquad \int_{\T^n}\chi(y,p)\dd y=0.
\end{equation}
For $H^*$, write $(\chi^*,\Hb^*)$ for the corresponding pair. The periodic cell theorem gives a unique normalized classical pair
for $H^*$ under \HYP\ and for $H$ under \eqref{eq:uniform-structure}, see \cite[Definition~1]{QSTY}. 

\subsection{Proof of the reduction theorem}\label{subsec:reduction}
For every $r>0$, choose a smooth momentum cutoff $\zeta$ equal to one
on $\overline B_{r+1}$ and zero outside $B_{r+2}$. Choose a smooth
convex radial function $J_r$, zero on $\overline B_r$, with globally
bounded Hessian and $D^2J_r\ge\mu_r I$ on $\R^n\setminus B_{r+1}$,
where $\mu_r>0$. For example, set $J_r(p)=\theta(|p|)$ with
$\theta(0)=\theta'(0)=0$ and
$\theta''(s)=2\sigma(s-r)$, where $\sigma$ is smooth and
nondecreasing, zero for $s\le0$ and one for $s\ge1$.
The radial and tangential Hessian eigenvalues are $\theta''(|p|)$
and $\theta'(|p|)/|p|$, respectively. The positive definiteness of
$H^*_{pp}$ on $\overline B_{r+2}$, together with the lower bound for
$D^2J_r$ on the transition annulus, shows that
$H^r=\zeta H^*+A_rJ_r\in C^3(\T^n\times\R^n)$, for sufficiently
large $A_r$, satisfies
\begin{equation}\label{eq:convex-extension-agreement}
 \begin{gathered}
 H^r=H^*\quad\text{on }\T^n\times\overline B_r,\\
 0<\lambda_r I\le H^r_{pp}\le\Lambda_r I,
 \qquad |H_y^r|\le C_r.
 \end{gathered}
\end{equation}
The modification and the positive constants $\lambda_r,\Lambda_r,C_r$
depend only on $H^*,n,r$. Write $\overline{H^r}$ for its effective
Hamiltonian. See also \cite{AF} for convex quadratic modifications,
and \cite[Section~2.1]{MTY} and \cite[Section~2.1]{TY} for the use of
modifications outside a bounded momentum set in periodic homogenization.

\begin{proof}[Proof of Theorem~\ref{thm:reduction}]
We first fix one modification of $H^*$, then compare the viscous
solutions, and finally identify the effective solutions.

\textit{Step 1. Fix the constants and the modification.}
Choose a nonnegative mollifier $\rho\in C_c^\infty(B_1)$ of integral
one, depending only on $n$, and set
\[
 h_K^*=\max_{y\in\T^n,\ |p|\le K}|H^*(y,p)|,\qquad
 c_\rho=\int_{\R^n}|z|\,|\Delta\rho(z)|\dd z,\qquad
 M=h_K^*+\tfrac12c_\rho K.
\]
For $f_\varepsilon=\rho_\varepsilon*g$, where
$\rho_\varepsilon(z)=\varepsilon^{-n}\rho(z/\varepsilon)$,
convolution gives
\begin{equation}\label{eq:localized-data}
 \|f_\varepsilon-g\|_\infty\le K\varepsilon,\qquad
 \|Df_\varepsilon\|_\infty\le K,\qquad
 \varepsilon\|\Delta f_\varepsilon\|_\infty\le c_\rho K.
\end{equation}
The last estimate follows by writing
$\Delta f_\varepsilon(x)=\int\Delta\rho_\varepsilon(z)
[g(x-z)-g(x)]\dd z$.

Write
\[
 \mathcal B(y,q)=H^*(y,q)^2+nH_y^*(y,q)\cdot q.
\]
By \HYP, choose $R=R(H^*,n,K)>K+1$ such that
\begin{equation}\label{eq:bernstein-threshold}
 \mathcal B(y,q)>2M^2+1\qquad(y\in\T^n,\ |q|\ge R).
\end{equation}
Using \eqref{eq:convex-extension-agreement}, fix $H=H^{2R}$.
In particular, $H=H^*$ on $\T^n\times\overline B_{2R}$ and
$H$ has all the properties in \eqref{eq:reduction-H}.
The quantities $M,R,H$ have now been fixed using only $H^*,n,K$;
none depends on $\varepsilon$ or a terminal time.

\textit{Step 2. Compare the viscous solutions.}
Let $\omega^\varepsilon$ be the smooth-data solution of
\[
 \omega_t^\varepsilon+H(x/\varepsilon,D\omega^\varepsilon)
 =\tfrac\varepsilon2\Delta\omega^\varepsilon,
 \qquad \omega^\varepsilon(\cdot,0)=f_\varepsilon.
\]
We claim that $\omega^\varepsilon$ also solves the equation with
$H^*$. It suffices to show that its gradient stays in $\overline B_R$.
First, \eqref{eq:localized-data} and $H=H^*$ on $\overline B_K$
make $f_\varepsilon-Mt$ and $f_\varepsilon+Mt$ a subsolution and
a supersolution, respectively. Comparison gives
$|\omega^\varepsilon(x,h)-f_\varepsilon(x)|\le Mh$ for $h\ge0$.
Applying comparison once more to time translates yields
\begin{equation}\label{eq:smooth-time-bound}
 |\omega^\varepsilon(x,t+h)-\omega^\varepsilon(x,t)|\le Mh,
 \qquad |\omega_t^\varepsilon|\le M\quad(t>0).
\end{equation}

Now set
\[
 v(y,\tau)=\varepsilon^{-1}\omega^\varepsilon(\varepsilon y,
              \varepsilon\tau),\qquad Z=\tfrac12|D_yv|^2.
\]
Then $D_yv=D_x\omega^\varepsilon$, $v_\tau=\omega_t^\varepsilon$,
and $v_\tau-\Delta_yv/2+H(y,D_yv)=0$.
Before the gradient reaches $2R$, this equation is also valid with
$H^*$ in place of $H$. Differentiation gives
\[
 Z_\tau-\tfrac12\Delta_yZ+H_p^*(y,D_yv)\cdot D_yZ
 =-\tfrac12|D_y^2v|_F^2-H_y^*(y,D_yv)\cdot D_yv.
\]
Using $\Delta_yv=2(v_\tau+H^*(y,D_yv))$ and
$2(a+b)^2\ge a^2-2b^2$, we have
\[
 \tfrac12|D_y^2v|_F^2
 \ge\frac{2}{n}(v_\tau+H^*(y,D_yv))^2
 \ge\frac{H^*(y,D_yv)^2-2M^2}{n}.
\]
Consequently,
\begin{equation}\label{eq:scaled-bernstein}
 Z_\tau-\tfrac12\Delta_yZ+H_p^*(y,D_yv)\cdot D_yZ
 \le\frac{2M^2-\mathcal B(y,D_yv)}{n}
 \le-\frac1n\quad\text{where }|D_yv|\ge R.
\end{equation}
This prevents $|D_yv|$ from exceeding $R$, since its initial value
is at most $K<R$. More precisely, on any finite time strip before
the first possible exit from $B_{2R}$, suppose that
$\sup Z>R^2/2$. For sufficiently small $\eta>0$, the function
$Z-\eta\sqrt{1+|y|^2}$ attains a maximum at a positive time where
$|D_yv|>R$. At this point, the left-hand side of
\eqref{eq:scaled-bernstein} is at least $-C_R\eta$, because
$H_p^*$ is bounded on $\T^n\times\overline B_{2R}$.
This contradicts \eqref{eq:scaled-bernstein} as $\eta\downarrow0$.
The bound $|D_yv|\le R<2R$ rules out the first exit and continues
for all times. Thus
\begin{equation}\label{eq:smooth-gradient-bound}
 \|D\omega^\varepsilon\|_{L^\infty(\R^n\times[0,\infty))}
 \le R\qquad(0<\varepsilon\le1).
\end{equation}
Both \eqref{eq:smooth-time-bound} and
\eqref{eq:smooth-gradient-bound} use the same $M,R$ for every
$\varepsilon$ and every terminal time.

Since $H=H^*$ on $\overline B_{2R}$,
$\omega^\varepsilon$ is a common solution of the two viscous
equations with initial datum $f_\varepsilon$.
Comparison and \eqref{eq:localized-data} therefore give
\[
 \|U^\varepsilon-\omega^\varepsilon\|_\infty\le K\varepsilon,
 \qquad
 \|u^\varepsilon-\omega^\varepsilon\|_\infty\le K\varepsilon,
\]
where both norms are over $\R^n\times[0,\infty)$.
The triangle inequality proves
\begin{equation}\label{eq:same-data-reduction}
 \|U^\varepsilon-u^\varepsilon\|_\infty\le2K\varepsilon.
\end{equation}

\textit{Step 3. Identify the effective solutions.}
Fix $|p|\le K$. At the maximum and minimum of
$\chi^*(\cdot,p)$, the cell equation gives
$|\Hb^*(p)|\le h_K^*\le M$. Set
$q=p+D_y\chi^*(y,p)$ and $Z=|q|^2/2$.
The stationary Bernstein calculation is
\[
 -\tfrac12\Delta_yZ+H_p^*(y,q)\cdot D_yZ
 =-\tfrac12|D_yq|_F^2-H_y^*(y,q)\cdot q.
\]
At a maximum of $Z$, use
$\operatorname{tr}D_yq=2(H^*(y,q)-\Hb^*(p))$ to obtain
\[
 0\ge\frac{2}{n}(H^*(y,q)-\Hb^*(p))^2
          +H_y^*(y,q)\cdot q
 \ge\frac{\mathcal B(y,q)-2|\Hb^*(p)|^2}{n}.
\]
Hence $\mathcal B(y,q)\le2M^2$ there, and
\eqref{eq:bernstein-threshold} gives
\[
 |p+D_y\chi^*(y,p)|<R
 \qquad(y\in\T^n,\ |p|\le K).
\]
The same corrector therefore solves the cell problem for $H$.
Uniqueness of the cell constant gives
\begin{equation}\label{eq:preserved-cell-pair}
 \Hb(p)=\Hb^*(p)\qquad(|p|\le K).
\end{equation}
Finally, spatial-translation comparison for the two autonomous
effective equations preserves the initial Lipschitz bound $K$.
All spatial test gradients of $U$ and $u$ therefore lie in
$\overline B_K$. By \eqref{eq:preserved-cell-pair}, they solve
the same effective Cauchy problem, and comparison yields $U=u$.
\end{proof}

\subsection{Proof of the regularity theorem on \texorpdfstring{$\Hb$ and $\Lb$}{Hbar and Lbar}}\label{subsec:cell}
\begin{proof}[Proof of Theorem~\ref{thm:cell}]
\textit{1. The cell pair and its dependence on $p$.}
Fix $\ell>n$. Since $H\in C^3$ and the cell corrector is classical,
bootstrapping \eqref{eq:cell} with the elliptic $L^\ell$ estimates gives
$\chi(\cdot,p)\in W^{5,\ell}(\T^n)\subset C^4(\T^n)$ for each fixed $p$.

Introduce the Banach spaces
\[
 X=\left\{f\in W^{2,\ell}(\T^n):\int_{\T^n}f\dd y=0\right\},
 \qquad Y=L^\ell(\T^n).
\]
The Sobolev space $W^{2,\ell}$ consists of periodic functions with
weak derivatives through order two in $L^\ell$; $X$ carries its
inherited norm. Consider
\[
 \mathcal F:X\times\R\times\R^n\longrightarrow Y,
 \qquad \mathcal F(f,d,p)=-\tfrac12\Delta_yf+H(y,p+D_yf)-d.
\]
The Sobolev embedding $W^{1,\ell}(\T^n)\subset C^0(\T^n)$ gives
$W^{2,\ell}(\T^n)\subset C^1(\T^n)$.
On a bounded neighborhood in $X\times\R^n$, all values of $p+D_yf$
therefore lie in a fixed compact ball in $\R^n$.
Taylor's formula shows that $N(f,p)=H(\cdot,p+D_yf)$ is $C^3$,
with derivatives, for $h_i\in X$ and $e_i\in\R^n$, given by
\[
 D^jN(f,p)[(h_1,e_1),\ldots,(h_j,e_j)]
 =D_p^jH(\cdot,p+D_yf)
 [e_1+D_yh_1,\ldots,e_j+D_yh_j],\qquad 1\le j\le3.
\]
Thus $\mathcal F$ is $C^3$.

At a cell pair, define
\[
 b_p(y)=H_p(y,p+D_y\chi(y,p)),\qquad
 \mathcal A_ph=-\tfrac12\Delta_yh+b_p\cdot D_yh.
\]
Then $D_{(f,d)}\mathcal F(h,a)=\mathcal A_ph-a$.
The map $(h,a)\mapsto-\Delta_yh/2-a$ is an isomorphism from
$X\times\R$ to $Y$, by the mean-zero periodic Poisson problem.
The drift term is compact into $Y$, because the embedding of
$W^{2,\ell}(\T^n)$ into $W^{1,\ell}(\T^n)$ is compact.
Hence $D_{(f,d)}\mathcal F$ is Fredholm of index zero.
If $\mathcal A_ph=a$, elliptic regularity makes $h$ classical;
its maximum and minimum give $a=0$, and then the strong maximum
principle and zero mean give $h=0$. Thus $D_{(f,d)}\mathcal F$ is
injective. Since its Fredholm index is zero, it is also surjective;
the bounded inverse theorem makes it an isomorphism. See
\cite[Sections~5.4--5.5]{Zeidler}.

The Banach-space $C^3$ implicit function theorem
\cite[Theorem~4.E, Section~4.8, pp.~250--251]{Zeidler} now gives
\begin{equation}\label{eq:cell-parameter-regularity}
 p\longmapsto(\chi(\cdot,p),\Hb(p))
 \quad\text{of class }C^3\text{ into }X\times\R.
\end{equation}
Uniqueness joins the local families. In particular, $\Hb\in C^3$.
We next verify $\chi\in C^2(\T^n\times\R^n)$, including the mixed
derivatives. By \eqref{eq:cell-parameter-regularity} and
$W^{2,\ell}\subset C^1$, the functions $\chi$, $D_y\chi$,
$D_p\chi$, $D_yD_p\chi$, and $D_p^2\chi$ are jointly continuous
in $(y,p)$. To treat $D_y^2\chi$, put
\[
 G_p=H(\cdot,p+D_y\chi(\cdot,p))-\Hb(p).
\]
The chain rule and continuity of $p\mapsto\chi(\cdot,p)$ in
$W^{2,\ell}$ show that $p\mapsto G_p$ is continuous in $W^{1,\ell}$.
Since $\Delta_y\chi(\cdot,p)=2G_p$ and $\chi$ has zero mean,
the periodic Poisson estimate gives, as $q\to p$,
\[
 \|\chi(\cdot,q)-\chi(\cdot,p)\|_{W^{3,\ell}}
 \le C\|G_q-G_p\|_{W^{1,\ell}}\longrightarrow0.
\]
The embedding $W^{3,\ell}\subset C^2$ now gives joint continuity
of $D_y^2\chi$. All derivatives of total order at most two are
therefore continuous on $\T^n\times\R^n$.

\textit{2. Positive Hessians and duality.}
Fix $p\in\R^n$ and $\xi\in\R^n\setminus\{0\}$, and put
\[
 P(y)=p+D_y\chi(y,p),\qquad
 a(y)=D_p\chi(y,p)\cdot\xi,\qquad
 b(y)=D^2_{pp}\chi(y,p)[\xi,\xi],\qquad
 Q(y)=\xi+D_ya(y).
\]
Twice differentiating \eqref{eq:cell} in the fixed direction
$\xi$, using \eqref{eq:cell-parameter-regularity}, gives
\begin{equation}\label{eq:effective-hessian-equation}
 \mathcal A_pb+Q^TH_{pp}(y,P)Q
   =\xi^TD^2\Hb(p)\xi.
\end{equation}
Here $\mathcal A_p$ is the operator defined in part~1, with
drift $H_p(y,P)$.
$b\in W^{2,\ell}$, and the coefficients and right-hand side
are locally H\"older continuous. Elliptic regularity makes $b$
classical, so the maximum principle applies.
Let $c=\xi^TD^2\Hb(p)\xi$. At a maximum of the periodic function
$b$, equation \eqref{eq:effective-hessian-equation} gives
$c\ge Q^TH_{pp}(y,P)Q\ge0$. If $c=0$, then
$\mathcal A_pb=-Q^TH_{pp}(y,P)Q\le0$ on the torus.
The strong maximum principle forces $b$ to be constant, so
\eqref{eq:effective-hessian-equation} and $H_{pp}>0$ imply
$Q\equiv0$. But periodicity gives
\[
 \int_{\T^n}Q(y)\dd y
 =\xi+\int_{\T^n}D_ya(y)\dd y=\xi\ne0,
\]
a contradiction. Thus $\xi^TD^2\Hb(p)\xi>0$ for every
$\xi\ne0$, which proves $D^2\Hb(p)>0$ directly.

At a maximum and a minimum of $\chi(\cdot,p)$,
$\min_yH(y,p)\le\Hb(p)\le\max_yH(y,p)$.
The quadratic lower and upper bounds for $H$ therefore pass to
$\Hb$, and convex duality gives quadratic growth for $\Lb$ as well.
In particular, $\Hb(p)-p\cdot q$ has a unique minimizer for every
$q$. Thus $D\Hb$ is bijective. Its derivative is everywhere
invertible, so the finite-dimensional inverse function theorem gives
$(D\Hb)^{-1}\in C^2$. Differentiating the dual formula yields
\begin{equation}\label{eq:dual-gradients}
 D\Lb=(D\Hb)^{-1},\qquad
 D^2\Lb(q)=\bigl[D^2\Hb(D\Lb(q))\bigr]^{-1}.
\end{equation}
It follows that $\Lb\in C^3$ and $D^2\Lb>0$. We will also use
\begin{equation}\label{eq:effective-identity}
 \Hb(D\Lb(q))=q\cdot D\Lb(q)-\Lb(q).
\end{equation}
\end{proof}

\section{Proof of the global estimate}\label{sec:proof}
\subsection{Preparations for the comparison calculation}\label{subsec:calculation}
We first record a heat-flow estimate, choose a ball containing the
minimizing velocities in the Hopf--Lax formula, and construct a
modification $\Phi$ of $\Lb$ with bounded gradient.

All radii and constants chosen in this subsection depend only on
$H,n,K$. Positive constants $c,C$ may change between
inequalities, with their dependence specified when used.
Vector and derivative-array norms are Euclidean. For matrices,
$|A|_F^2=\sum_{i,j}A_{ij}^2$, $A:B=\sum_{i,j}A_{ij}B_{ij}$,
$|A|_{\rm op}=\sup_{|e|=1}|Ae|$, and $\operatorname{tr}A=\sum_iA_{ii}$.
For a three-index array, $|T|_{HS}^2=\sum_{i,j,k}T_{ijk}^2$.
All indices range from $1$ to $n$, and $A^T$ denotes transpose.
A $C^j$ norm is the sum of the suprema of all derivatives through
order $j$ on the indicated compact set.

\begin{lemma}[Heat-flow estimate and minimizing velocities]\label{lem:basic-comparison}
Set $h_K=\max_{y\in\T^n,\ |p|\le K}|H(y,p)|$.
For all $x\in\R^n$ and $t\ge0$,
\begin{equation}\label{eq:heat-control}
 |u^\varepsilon(x,t)-g(x)|\le h_Kt+K\sqrt{n\varepsilon t}.
\end{equation}
For $t>0$, the effective solution is given by
\begin{equation}\label{eq:HL}
 u(x,t)=\min_{z\in\R^n}
       \left\{g(z)+t\Lb\!\left(\frac{x-z}{t}\right)\right\},\qquad t>0.
\end{equation}
The minimum is attained, and every minimizer $z_*$ satisfies
\begin{equation}\label{eq:minimizer-energy}
 \Lb\!\left(\frac{x-z_*}{t}\right)
 \le K\frac{|x-z_*|}{t}+\Lb(0).
\end{equation}
\end{lemma}
\begin{proof}
Let $G^\varepsilon$ be the heat flow with initial value $g$ and
diffusion coefficient $\varepsilon/2$. Convolution with the heat
kernel preserves the spatial Lipschitz bound $K$. Consequently,
$G^\varepsilon-h_Kt$ and $G^\varepsilon+h_Kt$ are, respectively,
a subsolution and a supersolution of the microscopic equation.
The heat kernel has second moment $n\varepsilon t$, so
\[
 |G^\varepsilon(x,t)-g(x)|\le K\sqrt{n\varepsilon t}.
\]
Comparison gives \eqref{eq:heat-control}.
The Hopf--Lax formula applies because $\Hb$ is convex and
superlinear; see \cite[Chapter~2]{TranBook}.
The expression minimized in \eqref{eq:HL} tends to $+\infty$ as
$|z|\to\infty$, because $g$ is bounded and $\Lb$ has quadratic growth.
Comparing a minimizer with the competitor $z=x$ gives
\[
 t\Lb\!\left(\frac{x-z_*}{t}\right)
 \le g(x)-g(z_*)+t\Lb(0)
 \le K|x-z_*|+t\Lb(0),
\]
which proves \eqref{eq:minimizer-energy}.
\end{proof}

We next choose the velocity ball used by both comparison profiles.
The exterior supersolution in Section~\ref{subsec:upper} will have
spatial slope $K+1$ and time slope
\begin{equation}\label{eq:exterior-slope}
 \gamma=1+\max_{y\in\T^n,\ |p|\le K+1}|H(y,p)|+\tfrac n2(K+1).
\end{equation}
Choose $R_1>0$ large enough that
\begin{equation}\label{eq:fixed-velocity-radius}
 \begin{gathered}
 R_1>2\left(|\Lb(0)|+\max_{|p|\le K+1}|\Hb(p)|\right)+1,
 \\
 \Lb(q)>(K+1)(|q|+1)+\gamma+1\quad(|q|\ge R_1).
 \end{gathered}
\end{equation}
Such a choice is possible by quadratic growth; fix, for example,
the least positive integer satisfying both inequalities.
The second inequality, together with
\eqref{eq:minimizer-energy}, implies that every Hopf--Lax minimizer
satisfies
\begin{equation}\label{eq:minimizer-displacement}
 |x-z_*|<R_1t.
\end{equation}
Indeed, for $|q|\ge R_1$ the lower bound in
\eqref{eq:fixed-velocity-radius} is strictly larger than
$K|q|+\Lb(0)$, by the first inequality there.
The first inequality also localizes a possible initial defect in
the lower comparison. The extra margin in the second gives the
strict ordering of the two upper-comparison branches at their interface.

For the lower profile, the velocity variable ranges over all of
$\R^n$. We therefore replace $\Lb$ by a function $\Phi$ that agrees
with it on $\overline B_{R_1}$ and has bounded gradient everywhere;
this keeps the momentum argument of $\chi$ in a fixed compact set.
We obtain $\Phi$ by adding a nonnegative convex barrier to $\Hb$
on a momentum ball $B_{R_2}$ and then taking its dual. This also
preserves the effective subsolution inequality, as proved next.

\begin{lemma}\label{lem:cap}
There exist $R_2=R_2(H,n,K)>K+1$ and $\Phi\in C^3(\R^n)$
such that
\begin{gather}
 \Phi\le\Lb,\qquad \Phi=\Lb\ \hbox{on }\overline B_{R_1},\qquad |D\Phi|<R_2,
                                                        \label{eq:cap-agreement}\\
 \Phi(q)\ge (K+1)|q|-\sup_{|p|\le K+1}|\Hb(p)|,                 \label{eq:cap-growth}\\
 \Phi(q)-q\cdot D\Phi(q)+\Hb(D\Phi(q))\le0,\label{eq:cap-effective}\\
 \sup_{q\in\R^n}\left(|D^2\Phi(q)|_F+|D^3\Phi(q)|_{HS}
                   +|D^2\Phi(q)q|\right)\le C(H,n,K).    \label{eq:cap-derivatives}
\end{gather}
\end{lemma}
\begin{proof}
Choose
\begin{equation}\label{eq:fixed-parameter-radius}
 R_2=K+4+\max_{|q|\le R_1}|D\Lb(q)|.
\end{equation}
Then $D\Lb(\overline B_{R_1})$ and $\overline B_{K+1}$ both lie in
$B_{R_2-1}$. Define the radial barrier
\[
 \psi(r)=\frac{(r_+)^4}{1-r}\quad(r<1),\qquad
 j(p)=\psi(|p|-R_2+1)\quad(|p|<R_2),
\]
where $r_+=\max\{r,0\}$. The function $\psi$ is $C^3$,
nonnegative, nondecreasing, and convex. Indeed, on $(0,1)$ it is
the product of the nonnegative, nondecreasing convex functions
$r^4$ and $(1-r)^{-1}$, and its first three derivatives vanish
at zero. Hence $j\in C^3(B_{R_2})$ is convex, vanishes on
$\overline B_{R_2-1}$, and tends to $+\infty$ at $\partial B_{R_2}$.
There is no regularity issue at $p=0$, since $j$ vanishes nearby.
Define
\begin{equation}\label{eq:modified-transform}
 \widehat H(p)=\Hb(p)+j(p)\quad(p\in B_{R_2}),\qquad
 \Phi(q)=\max_{p\in B_{R_2}}\{p\cdot q-\widehat H(p)\}.
\end{equation}
For each $q\in\R^n$, the function
$p\mapsto p\cdot q-\widehat H(p)$ tends to $-\infty$ as
$p$ approaches $\partial B_{R_2}$. Strict convexity of $\widehat H$
therefore gives a unique interior maximizer.
Its equation is $D\widehat H(p)=q$. Since
$\widehat H\in C^3$ and $D^2\widehat H>0$, the inverse function
theorem makes this maximizer a $C^2$ function of $q$.
Differentiation of \eqref{eq:modified-transform} identifies it
with $D\Phi(q)$, so $\Phi\in C^3$ and $|D\Phi|<R_2$. In particular,
\begin{equation}\label{eq:modified-identity}
 D\widehat H(D\Phi(q))=q,\qquad
 \widehat H(D\Phi(q))=q\cdot D\Phi(q)-\Phi(q),
\end{equation}
and
\[
 D\Phi(q)=(D\widehat H)^{-1}(q),\qquad
 D^2\Phi(q)=(D^2\widehat H(D\Phi(q)))^{-1}.
\]
Since $\widehat H\ge\Hb$, one has $\Phi\le\Lb$.
For $|q|\le R_1$, the maximizer $D\Lb(q)$ lies where $j=0$, proving
equality. Choosing $p=(K+1)q/|q|$ in \eqref{eq:modified-transform} proves \eqref{eq:cap-growth};
at $q=0$ use $p=0$. The second identity in
\eqref{eq:modified-identity} also gives
$\Phi-q\cdot D\Phi+\Hb(D\Phi)=-j(D\Phi)\le0$,
which proves \eqref{eq:cap-effective}.

It remains to prove the global derivative bounds.
Write $j(p)=\varphi(\rho)$ with $\rho=|p|$, and put $d=R_2-\rho$.
For $0<d<1$, the explicit formula is
\[
 \varphi(\rho)=\frac{(1-d)^4}{d}
 =d^{-1}-4+6d-4d^2+d^3.
\]
Thus, for $0<d\le1/2$,
\[
 c d^{-2}\le\varphi'(\rho)\le C d^{-2},\qquad
 c d^{-3}\le\varphi''(\rho)\le C d^{-3},\qquad
 |\varphi'''(\rho)|\le C d^{-4}.
\]
The radial and tangential eigenvalues of $D^2j$ are
$\varphi''(\rho)$ and $\varphi'(\rho)/\rho$, respectively.
Since $\rho\ge R_2-1/2>0$ in this region, the radial derivative
formulas and the $C^3$ bound for $\Hb$ on $\overline B_{R_2}$ give
\[
 \left|(D^2\widehat H(p))^{-1}\right|_{\rm op}\le Cd^2,\qquad
 |D\widehat H(p)|\le Cd^{-2},\qquad
 |D^3\widehat H(p)|_{HS}\le Cd^{-4}.
\]
These three bounds hold on all of $B_{R_2}$ after increasing
$C=C(H,n,K)$: on $|p|\le R_2-1/2$, use compactness and
$D^2\widehat H\ge D^2\Hb\ge cI$.
At $p=D\Phi(q)$, put $A=(D^2\widehat H(p))^{-1}=D^2\Phi(q)$.
Differentiating the inverse-Hessian formula yields
\[
 D^3\Phi(q)[h_1,h_2,h_3]
 =-D^3\widehat H(p)[Ah_1,Ah_2,Ah_3].
\]
Here brackets denote evaluation in the indicated directions.
Using also $q=D\widehat H(p)$, we obtain
\[
 |D^2\Phi(q)|_F\le Cd^2,\qquad
 |D^3\Phi(q)|_{HS}\le Cd^6d^{-4}=Cd^2,\qquad
 |D^2\Phi(q)q|\le Cd^2d^{-2}=C.
\]
Since $0<d\le R_2$, these are the bounds in
\eqref{eq:cap-derivatives}.
\end{proof}

With $R_1,R_2$ now fixed, Theorem~\ref{thm:cell} and
compactness give a single constant $C=C(H,n,K)$ such that
\begin{equation}\label{eq:compact-derivative-bounds}
 \|\chi\|_{C^2(\T^n\times\overline B_{R_2})}
 +\|\Hb\|_{C^3(\overline B_{R_2})}
 +\|\Lb\|_{C^3(\overline B_{R_1})}
 +\max_{y\in\T^n,\ |p|\le R_2}|H_p(y,p+D_y\chi(y,p))|\le C.
\end{equation}
In the last term, the cell gradient $p+D_y\chi(y,p)$ ranges over
a compact set because it is continuous; no additional radius is needed.
These bounds use the regularity already proved, with $H$ fixed
independently of $\varepsilon$.

For a function $f(z,s)$ that is twice continuously differentiable
in the spatial variable $z$ and once in the time variable $s$, define
\begin{equation}\label{eq:comparison-operator}
 \Pe f=f_s+H(z/\varepsilon,Df)-\frac{\varepsilon}{2}\Delta f.
\end{equation}
Spatial derivatives without a subscript are in $z$.
For a comparison function $v\in C^3$, define
\begin{equation}\label{eq:corrector-approximation}
 w(z,s)=v(z,s)+\varepsilon\chi(z/\varepsilon,Dv(z,s)).
\end{equation}
The next two lemmas calculate the equation error of $w$ and then
estimate it. In both lemmas we write
\[
 p=Dv,\qquad S=D^2v,
\]
and evaluate $\chi$ and all its derivatives at $(z/\varepsilon,p)$.

\begin{lemma}\label{lem:calculation}
Assume \eqref{eq:uniform-structure}. Let $v\in C^3$ on an open space-time set and let
$w$ be given by \eqref{eq:corrector-approximation}. Then
\begin{align}
 \Pe w={}&v_s+\Hb(p)+\varepsilon\chi_p\cdot Dv_s
 +H(z/\varepsilon,p+\chi_y+\varepsilon S\chi_p)
             -H(z/\varepsilon,p+\chi_y)\notag\\
 &-\frac{\varepsilon}{2}\operatorname{tr}S-\varepsilon\chi_{yp}:S
 -\frac{\varepsilon^2}{2}
       \left[\chi_{pp}:(SS^T)+\chi_p\cdot D\Delta v\right].
                                                    \label{eq:residual}
\end{align}
\end{lemma}
\begin{proof}
The chain rule gives
\begin{gather*}
 Dw=p+\chi_y+\varepsilon S\chi_p,\qquad
 w_s=v_s+\varepsilon\chi_p\cdot Dv_s,\\
 \Delta w=\operatorname{tr}S+\varepsilon^{-1}\Delta_y\chi
       +2\chi_{yp}:S+\varepsilon\chi_{pp}:(SS^T)
       +\varepsilon\chi_p\cdot D\Delta v.
\end{gather*}
The term $D\Delta v$ is where three spatial derivatives of $v$
enter the calculation. The corrector itself needs only joint $C^2$
regularity. Substitute these identities into \eqref{eq:comparison-operator}.
The cell equation gives
$H(z/\varepsilon,p+\chi_y)-\tfrac12\Delta_y\chi=\Hb(p)$, which proves
\eqref{eq:residual}.
\end{proof}

\begin{lemma}\label{lem:residual}
Use the fixed $R_2$ above. Let $\kappa>0$ and let $r>0$ be a real number.
Let $v\in C^3$ and let $w$ be given by
\eqref{eq:corrector-approximation}.
At any point where
\begin{equation}\label{eq:comparison-derivatives}
 |Dv|\le R_2,\qquad
 |D^2v|_F+|Dv_s|\le\frac{\kappa}{r},\qquad
 |D^3v|_{HS}\le\frac{\kappa}{r^2},
\end{equation}
one has
\begin{align}
 |w-v|&\le C\varepsilon,\label{eq:corrector-value}\\
 |\Pe w-(v_s+\Hb(Dv))|
 &\le C\left(\frac{\varepsilon}{r}
                     +\frac{\varepsilon^2}{r^2}\right),
                                             \label{eq:residual-bound}
\end{align}
where $C=C(H,n,K,\kappa)$ is independent of
$\varepsilon,r,z,s$ and of the choice of $v$ satisfying
\eqref{eq:comparison-derivatives}.
\end{lemma}
\begin{proof}
By \eqref{eq:compact-derivative-bounds}, the functions
\[
 \chi,\quad \chi_p,\quad \chi_{yp},\quad \chi_{pp},\quad
 H_p(z/\varepsilon,p+\chi_y)
\]
are uniformly bounded whenever $|p|\le R_2$.
Periodicity makes these bounds independent of $z/\varepsilon$.
This immediately gives \eqref{eq:corrector-value}.

Taylor's formula and $H_{pp}\le\Lambda I$ give
\begin{align*}
 |H(z/\varepsilon,p+\chi_y+\varepsilon S\chi_p)
                   -H(z/\varepsilon,p+\chi_y)|
 &\le\varepsilon|H_p(z/\varepsilon,p+\chi_y)|\,|S\chi_p|
             +\frac{\Lambda}{2}\varepsilon^2|S\chi_p|^2\\
 &\le C\bigl(\varepsilon|S|_F+\varepsilon^2|S|_F^2\bigr).
\end{align*}
The global bound for $H_{pp}$ controls the increment
$\varepsilon S\chi_p$, which need not remain in the compact set
used for the base cell gradient.
The other terms in \eqref{eq:residual} are bounded using
\[
 |\operatorname{tr}S|\le\sqrt n\,|S|_F,\qquad
 |\chi_{pp}:(SS^T)|\le|\chi_{pp}|_{\rm op}|S|_F^2,\qquad
 |D\Delta v|\le\sqrt n\,|D^3v|_{HS}.
\]
Consequently, with $C=C(H,n,K)$ at this stage,
\begin{equation}\label{eq:structural-error}
 |\Pe w-(v_s+\Hb(Dv))|
 \le C\varepsilon\bigl(|Dv_s|+|D^2v|_F\bigr)
 +C\varepsilon^2\bigl(|D^2v|_F^2+|D^3v|_{HS}\bigr).
\end{equation}
Substitution of \eqref{eq:comparison-derivatives} proves
\eqref{eq:residual-bound}, after allowing $C$ to depend on $\kappa$.
\end{proof}

We now collect the bounds needed for the lower and upper comparison
profiles.
\begin{lemma}\label{lem:profiles}
There is $C=C(H,n,K)$ with the following properties.
For $a\in\R^n$, $c\in\R$, and $\tau>0$, define
\[
 f_-(z,s)=c-(\tau-s)\Phi\!\left(\frac{a-z}{\tau-s}\right)
 \quad(s<\tau),\qquad
 f_+(z,s)=c+(s+\tau)\Lb\!\left(\frac{z-a}{s+\tau}\right)
 \quad(s>-\tau),
\]
and, for $0<\varepsilon\le1$, set
$w_\pm=f_\pm+\varepsilon\chi(z/\varepsilon,Df_\pm)$.
For the minus profile everywhere, and for the plus profile where
$|z-a|<R_1(s+\tau)$, one has
\begin{gather}
 (f_-)_s+\Hb(Df_-)\le0,\qquad
 (f_+)_s+\Hb(Df_+)=0,\label{eq:profile-effective}\\
 |w_\pm-f_\pm|\le C\varepsilon,\label{eq:profile-corrector}\\
 \Pe w_-\le C\left(\frac\varepsilon{\tau-s}
                   +\frac{\varepsilon^2}{(\tau-s)^2}\right),\qquad
 |\Pe w_+|\le C\left(\frac\varepsilon{s+\tau}
                   +\frac{\varepsilon^2}{(s+\tau)^2}\right).
 \label{eq:profile-residuals}
\end{gather}
Also, for $b\in\R^n$ and $r_1,r_2>0$ satisfying
$|b|\le R_1\min\{r_1,r_2\}$,
\begin{equation}\label{eq:time-shift-estimate}
 \left|r_1\Lb(b/r_1)-r_2\Lb(b/r_2)\right|
 \le C|r_1-r_2|.
\end{equation}
These constants are independent of $a,c,\tau,\varepsilon,r_1,r_2,b$.
\end{lemma}
\begin{proof}
Put $r=\tau-s$, $q=(a-z)/r$ for the minus profile.
By \eqref{eq:cap-effective},
\[
 (f_-)_s+\Hb(Df_-)
 =\Phi(q)-q\cdot D\Phi(q)+\Hb(D\Phi(q))
 \le0.
\]
For the plus profile put $r=s+\tau$, $q=(z-a)/r$;
\eqref{eq:effective-identity} gives
$(f_+)_s+\Hb(Df_+)=0$. The spatial derivatives are
\begin{align*}
 Df_-&=D\Phi(q),&D^2f_-&=-r^{-1}D^2\Phi(q),
 &D(f_-)_s&=r^{-1}D^2\Phi(q)q,\\
 Df_+&=D\Lb(q),&D^2f_+&=r^{-1}D^2\Lb(q),
 &D(f_+)_s&=-r^{-1}D^2\Lb(q)q.
\end{align*}
Their third derivative norms are respectively
$r^{-2}|D^3\Phi(q)|_{HS}$ and $r^{-2}|D^3\Lb(q)|_{HS}$.
Equation \eqref{eq:compact-derivative-bounds} and Lemma~\ref{lem:cap} give
\eqref{eq:comparison-derivatives}, with the fixed radius $R_2$
and a constant $\kappa=C(H,n,K)$ on the stated regions.
Lemma~\ref{lem:residual} proves
\eqref{eq:profile-corrector}--\eqref{eq:profile-residuals}.
Finally,
\[
 \frac{\dd}{\dd r}\,[r\Lb(b/r)]
 =-\Hb(D\Lb(b/r)).
\]
When $r$ lies between $r_1$ and $r_2$, $|b/r|\le R_1$.
Compactness and integration give
\eqref{eq:time-shift-estimate}.
\end{proof}

The order of choices is
\begin{equation}\label{eq:radius-chain}
 (H^*,n,K)\ \longrightarrow\ R\ \longrightarrow\ H
 \ \longrightarrow\ (R_1,R_2,\Phi)\ \longrightarrow\ C(H,n,K).
\end{equation}
The fixed $n,K$ are retained throughout. The radii $R_1,R_2$ may
exceed $R$; their bounds and all compact norms above are determined
by $H,n,K$, independently of $\varepsilon,x,t,\delta$.
Only the Lipschitz constant of $g$ enters the comparison estimates.

\subsection{Lower bound}\label{subsec:lower}
\begin{proposition}\label{prop:lower}
Let $u^\varepsilon,u$ be as in Theorem~\ref{thm:reduction}.
There is $C=C(H,n,K)$ such that, for all
$x\in\R^n$, $0<\varepsilon\le1$, and $0<\delta\le t$,
\begin{equation}\label{eq:master-lower}
 u(x,t)-u^\varepsilon(x,t)
 \le C\left[\delta+\varepsilon
       +\varepsilon\log(1+t/\delta)+\frac{\varepsilon^2}{\delta}\right].
\end{equation}
\end{proposition}
\begin{proof}
\textbf{Construction.}
Fix $(x,t)$; the comparison variables are $z\in\R^n$ and $0\le s\le t$.
Use $R_1$ from \eqref{eq:fixed-velocity-radius}
and $\Phi$ from Lemma~\ref{lem:cap}.
In the introductory function $v_1$, replace $\Lb$
by $\Phi$ and regularize the remaining time. Thus, for $0\le s\le t$, define
\begin{equation}\label{eq:negative-cone}
 v_{1,\delta}(z,s)=u(x,t)+\delta\Lb(0)
           -(t+\delta-s)\Phi\!\left(\frac{x-z}{t+\delta-s}\right).
\end{equation}
The constant $\delta\Lb(0)$ makes the value at the target exact:
\[
 v_{1,\delta}(x,t)=u(x,t)+\delta\Lb(0)-\delta\Phi(0)=u(x,t).
\]
For a constant $C_0=C_0(H,n,K)$ to be chosen below, set
\begin{equation}\label{eq:lower-corrected}
 \begin{aligned}
 w_1(z,s)={}&v_{1,\delta}(z,s)
 +\varepsilon\chi(z/\varepsilon,Dv_{1,\delta}(z,s))
 -C_0(\delta+\varepsilon)\\
 &-C_0\int_0^s\left(\frac{\varepsilon}{t+\delta-\sigma}
                  +\frac{\varepsilon^2}{(t+\delta-\sigma)^2}\right)\dd\sigma.
 \end{aligned}
\end{equation}
The constant shift will enforce the initial comparison, and the
time integral will absorb the equation error.

\smallskip
\noindent\textbf{Checking the subsolution property.}
Lemma~\ref{lem:profiles}, with $\tau=t+\delta$, gives
\begin{equation}\label{eq:negative-subsolution}
 \partial_s v_{1,\delta}+\Hb(Dv_{1,\delta})\le0,
\end{equation}
and
\[
 \Pe\bigl(v_{1,\delta}+\varepsilon\chi(z/\varepsilon,Dv_{1,\delta})\bigr)
 \le C\left(\frac{\varepsilon}{t+\delta-s}
                    +\frac{\varepsilon^2}{(t+\delta-s)^2}\right).
\]
Taking $C_0\ge C$ in \eqref{eq:lower-corrected} gives $\Pe w_1\le0$
on $\R^n\times(0,t)$. Increasing $C_0$ in the boundary check below
preserves this inequality.

\smallskip
\noindent\textbf{Checking the boundary conditions.}
We show that
\begin{equation}\label{eq:initial-defect}
 v_{1,\delta}(z,0)\le g(z)+C\delta\qquad(z\in\R^n).
\end{equation}
There is nothing to check where $v_{1,\delta}(z,0)\le g(z)$.
At any other point, use
$u(x,t)\le g(x)+t\Lb(0)$, the Lipschitz bound for $g$, and
\eqref{eq:cap-growth}. They give
\[
 v_{1,\delta}(z,0)-g(z)
 \le -|x-z|+(t+\delta)
        \left(|\Lb(0)|+\sup_{|p|\le K+1}|\Hb(p)|\right)
 <-|x-z|+tR_1.
\]
Here $\delta\le t$ and \eqref{eq:fixed-velocity-radius} were used.
Consequently $|x-z|<tR_1$, and
\[
 \left|\frac{x-z}{\sigma}\right|<R_1
       \qquad(t\le\sigma\le t+\delta).
\]
On this set $\Phi=\Lb$. Using
$u(x,t)\le g(z)+t\Lb((x-z)/t)$ from \eqref{eq:HL}, we obtain
\[
 v_{1,\delta}(z,0)-g(z)
 \le \delta\Lb(0)
       +t\Lb\!\left(\frac{x-z}{t}\right)
       -(t+\delta)\Lb\!\left(\frac{x-z}{t+\delta}\right).
\]
The time-shift estimate \eqref{eq:time-shift-estimate}, with
$b=x-z$, $r_1=t$, and $r_2=t+\delta$, bounds the last two terms
by $C\delta$. This proves \eqref{eq:initial-defect}.
Now fix $C_0$ large enough also to dominate the constant in
\eqref{eq:initial-defect} and the corrector bound
\eqref{eq:profile-corrector}. Then
\[
 w_1(z,0)\le g(z)\qquad(z\in\R^n).
\]

By \eqref{eq:cap-growth} and the corrector bound
\eqref{eq:profile-corrector}, $w_1(z,s)\to-\infty$ as
$|z-x|\to\infty$, uniformly for $0\le s\le t$.
Since $u^\varepsilon$ is bounded on this time strip,
$w_1\le u^\varepsilon$ on $\partial B_L(x)\times[0,t]$
for all sufficiently large $L$.

\smallskip
\noindent\textbf{Comparison.}
Comparison on $B_L(x)\times[0,t]$ gives
$w_1(x,t)\le u^\varepsilon(x,t)$.
The radius $L$ does not enter the resulting estimate.
At the target, $v_{1,\delta}(x,t)=u(x,t)$. The remaining time integral is
\begin{equation}\label{eq:lower-integral}
 \int_0^t\left(\frac{\varepsilon}{t+\delta-s}
             +\frac{\varepsilon^2}{(t+\delta-s)^2}\right)\dd s
 =\varepsilon\log(1+t/\delta)
   +\varepsilon^2\left(\frac1\delta-\frac1{t+\delta}\right)
 \le\varepsilon\log(1+t/\delta)+\frac{\varepsilon^2}{\delta}.
\end{equation}
Together with \eqref{eq:lower-corrected} and
\eqref{eq:profile-corrector}, this proves \eqref{eq:master-lower}.
\end{proof}

\subsection{Upper bound}\label{subsec:upper}
\begin{proposition}\label{prop:upper}
Let $u^\varepsilon,u$ be as in Theorem~\ref{thm:reduction}.
There is $C=C(H,n,K)$ such that, for all
$x\in\R^n$, $0<\varepsilon\le1$, $t>0$, and $\delta\ge\varepsilon$,
\begin{equation}\label{eq:master-upper}
 u^\varepsilon(x,t)-u(x,t)
 \le C\left[\delta+\varepsilon
       +\varepsilon\log(1+t/\delta)+\frac{\varepsilon^2}{\delta}\right].
\end{equation}
\end{proposition}
\begin{proof}
\textbf{Construction.}
Fix $(x,t)$; the comparison variables are $z\in\R^n$ and $0\le s\le t$.
Choose a Hopf--Lax minimizer $z_*$ at $(x,t)$. By
\eqref{eq:HL} and \eqref{eq:minimizer-displacement},
\begin{equation}\label{eq:upper-minimizer}
 u(x,t)=g(z_*)+t\Lb\!\left(\frac{x-z_*}{t}\right),\qquad
 |x-z_*|<R_1t.
\end{equation}
For each comparison time $s$, define the inner region
\[
 \Omega_\delta(s)=B_{R_1(s+\delta)}(z_*),\qquad
 q_\delta(z,s)=\frac{z-z_*}{s+\delta}.
\]
On this region $|q_\delta|<R_1$, so the corrector and its
required derivatives are controlled by Section~\ref{subsec:calculation}.
We use a corrected effective profile there, and a separate
supersolution of linear spatial growth outside. The minimum used
in the inner region will coincide with the exterior branch near
the moving interface.

Set
\begin{equation}\label{eq:upper-profile}
 v_{2,\delta}(z,s)=u(x,t)+(s+\delta)\Lb(q_\delta(z,s))
 -(t+\delta)\Lb\!\left(\frac{x-z_*}{t+\delta}\right).
\end{equation}
Thus $v_{2,\delta}(x,t)=u(x,t)$. For a constant
$C_0=C_0(H,n,K)$ to be fixed below, define on
$\overline\Omega_\delta(s)$ the inner branch
\begin{equation}\label{eq:upper-corrected}
 \begin{aligned}
 w_2(z,s)={}&v_{2,\delta}(z,s)
 +\varepsilon\chi(z/\varepsilon,Dv_{2,\delta}(z,s))
 +C_0(\delta+\varepsilon)\\
 &+C_0\int_0^s\left(\frac{\varepsilon}{\sigma+\delta}
                  +\frac{\varepsilon^2}{(\sigma+\delta)^2}\right)\dd\sigma.
 \end{aligned}
\end{equation}
Using $\gamma$ from \eqref{eq:exterior-slope}, define the exterior
branch on all of $\R^n\times[0,t]$ by
\begin{equation}\label{eq:upper-exterior}
 \widehat w(z,s)=g(z_*)+(K+1)\sqrt{|z-z_*|^2+\delta^2}+\gamma s.
\end{equation}
Our comparison function is
\begin{equation}\label{eq:upper-patch}
 W_\delta(z,s)=
 \begin{cases}
  \min\{w_2(z,s),\widehat w(z,s)\},&z\in\Omega_\delta(s),\\
  \widehat w(z,s),&z\notin\Omega_\delta(s).
 \end{cases}
\end{equation}
We now choose $C_0$ and verify the supersolution and boundary
properties of this function.

\smallskip
\noindent\textbf{Checking the supersolution property.}
Lemma~\ref{lem:profiles}, with $\tau=\delta$, gives in the inner region
\[
 \partial_s v_{2,\delta}+\Hb(Dv_{2,\delta})=0,\qquad
 \left|\Pe\bigl(v_{2,\delta}
       +\varepsilon\chi(z/\varepsilon,Dv_{2,\delta})\bigr)\right|
 \le C\left(\frac\varepsilon{s+\delta}
       +\frac{\varepsilon^2}{(s+\delta)^2}\right).
\]
For $C_0$ sufficiently large, the positive time integral in
\eqref{eq:upper-corrected} therefore makes $\Pe w_2\ge0$
in the inner region.
For the exterior branch, direct differentiation gives
\[
 |D\widehat w|\le K+1,\qquad
 \Delta\widehat w\le\frac{n(K+1)}{\delta}.
\]
Since $\varepsilon\le\delta$, the choice of $\gamma$ implies
$\Pe\widehat w\ge1$ everywhere.
Within the inner region, the minimum of these two supersolutions
is again a viscosity supersolution: a smooth test function touching
the minimum from below touches an active branch from below.

To check the moving interface, we first control the normalization
of $w_2$.
The time-shift estimate \eqref{eq:time-shift-estimate}, with
$b=x-z_*$, $r_1=t$, and $r_2=t+\delta$, gives
\begin{equation}\label{eq:upper-target}
 \left|u(x,t)-g(z_*)
 -(t+\delta)\Lb\!\left(\frac{x-z_*}{t+\delta}\right)\right|
 \le C\delta.
\end{equation}
Taking $C_0$ also larger than the constant in this estimate and
the corrector bound \eqref{eq:profile-corrector} gives, throughout
the closed inner region,
\begin{equation}\label{eq:upper-branch-domination}
 w_2(z,s)\ge g(z_*)+(s+\delta)\Lb(q_\delta(z,s)).
\end{equation}
At $z\in\partial\Omega_\delta(s)$ we have $|q_\delta|=R_1$ and
\[
 \begin{aligned}
 \widehat w(z,s)
 &\le g(z_*)+(s+\delta)\bigl[(K+1)(R_1+1)+\gamma\bigr]\\
 &<g(z_*)+(s+\delta)\Lb(q_\delta(z,s))
 \le w_2(z,s),
 \end{aligned}
\]
by \eqref{eq:fixed-velocity-radius} and
\eqref{eq:upper-branch-domination}.
This strict inequality and continuity show that $W_\delta$ equals
$\widehat w$ in a neighborhood of the moving interface. Hence
$W_\delta$ is a continuous viscosity supersolution on
$\R^n\times(0,t)$.

\smallskip
\noindent\textbf{Checking the boundary conditions.}
For the initial boundary, the definition of the Legendre transform
implies
\[
 \Lb(q)\ge K|q|-\sup_{|p|\le K}|\Hb(p)|.
\]
Together with \eqref{eq:upper-target}, this gives
\[
 v_{2,\delta}(z,0)\ge g(z_*)+K|z-z_*|-C\delta
                         \ge g(z)-C\delta.
\]
Fix $C_0$ large enough also to absorb this error and the corrector's
value. Then $w_2(z,0)\ge g(z)$ on $\overline\Omega_\delta(0)$,
and all preceding supersolution and interface inequalities remain
valid. The Lipschitz bound for $g$ gives
$\widehat w(z,0)\ge g(z)$ everywhere. Thus $W_\delta(z,0)\ge g(z)$.

The exterior branch $\widehat w(z,s)$ tends to $+\infty$ as
$|z-z_*|\to\infty$, uniformly for $0\le s\le t$.
Since $u^\varepsilon$ is bounded on this strip, we can choose
$L>R_1(t+\delta)$ sufficiently large that
$W_\delta=\widehat w\ge u^\varepsilon$
on $\partial B_L(z_*)\times[0,t]$.

\smallskip
\noindent\textbf{Comparison.}
Comparison on $B_L(z_*)\times[0,t]$ gives
$u^\varepsilon\le W_\delta$ there. The target $(x,t)$ is in the
inner region by \eqref{eq:upper-minimizer}, and hence
\[
 u^\varepsilon(x,t)\le W_\delta(x,t)\le w_2(x,t).
\]
The radius of the comparison ball does not enter the estimates.
Finally,
\begin{equation}\label{eq:upper-integral}
 \int_0^t\left(\frac{\varepsilon}{s+\delta}
             +\frac{\varepsilon^2}{(s+\delta)^2}\right)\dd s
 =\varepsilon\log(1+t/\delta)
   +\varepsilon^2\left(\frac1\delta-\frac1{t+\delta}\right)
 \le\varepsilon\log(1+t/\delta)+\frac{\varepsilon^2}{\delta}.
\end{equation}
Together with $v_{2,\delta}(x,t)=u(x,t)$,
\eqref{eq:upper-corrected}, and \eqref{eq:profile-corrector}, this
proves \eqref{eq:master-upper}.
\end{proof}

\subsection{Conclusion}\label{subsec:conclusion}
\begin{proof}[Proof of Theorem~\ref{thm:global}]
For $t\ge\varepsilon$, take $\delta=\varepsilon$ in
Propositions~\ref{prop:lower} and \ref{prop:upper}. This gives
\begin{equation}\label{eq:uniform-rate}
 |u^\varepsilon(x,t)-u(x,t)|
 \le C(H,n,K)\varepsilon[1+\log(1+t/\varepsilon)].
\end{equation}
For $0\le t\le\varepsilon$, the heat bound \eqref{eq:heat-control} and
the Hopf--Lax formula \eqref{eq:HL} give
\[
 |u^\varepsilon(x,t)-g(x)|
 \le h_Kt+K\sqrt{n\varepsilon t},\qquad
 |u(x,t)-g(x)|\le Ct.
\]
The second bound follows from
$\inf_q\{\Lb(q)-K|q|\}>-\infty$ and $\Lb(0)<\infty$.
Thus \eqref{eq:uniform-rate} also holds at these times.
All constants depend only on $H,n,K$, by Section~\ref{subsec:calculation}.

Finally, Theorem~\ref{thm:reduction} yields
\[
 |U^\varepsilon(x,t)-U(x,t)|
 \le2K\varepsilon+|u^\varepsilon(x,t)-u(x,t)|
 \le C(H^*,n,K)\varepsilon[1+\log(1+t/\varepsilon)],
\]
because $H$ is fixed by $H^*,n,K$, independently of $\varepsilon$.
\end{proof}

\section*{Acknowledgments} GPT-6 Astra is used to assist in this study, while the author assumes full responsibility for the final content. Specifically, the initially incomplete proof was identified by this large language model after the author shared with them the idea of using the Hopf–Lax formula. This work is partially supported by the NSFC Grant No.\,12571220 and by the New Cornerstone Investigator Program 100001127.


\begin{thebibliography}{99}
\small
\setlength{\itemsep}{1.5pt}
\setlength{\parsep}{0pt}

\bibitem{AF}
A.~Abbondandolo and A.~Figalli,
\emph{High action orbits for Tonelli Lagrangians and superlinear
Hamiltonians on compact configuration spaces},
J. Differential Equations \textbf{234} (2007), no.~2, 626--653.
\href{https://doi.org/10.1016/j.jde.2006.10.015}{doi:10.1016/j.jde.2006.10.015}.

\bibitem{AT}
S.~N.~Armstrong and H.~V.~Tran,
\emph{Viscosity solutions of general viscous Hamilton--Jacobi equations},
Math. Ann. \textbf{361} (2015), nos.~3--4, 647--687.
\href{https://doi.org/10.1007/s00208-014-1088-5}{doi:10.1007/s00208-014-1088-5}.

\bibitem{Burago}
D.~Burago,
\emph{Periodic metrics},
Adv. Soviet Math. \textbf{9} (1992), 205--210.
\href{https://doi.org/10.1090/advsov/009/10}{doi:10.1090/advsov/009/10}.

\bibitem{CCM}
F.~Camilli, A.~Cesaroni, and C.~Marchi,
\emph{Homogenization and vanishing viscosity in fully nonlinear
elliptic equations: rate of convergence estimates},
Adv. Nonlinear Stud. \textbf{11} (2011), no.~2, 405--428.
\href{https://doi.org/10.1515/ans-2011-0210}{doi:10.1515/ans-2011-0210}.

\bibitem{CDI}
I.~Capuzzo-Dolcetta and H.~Ishii,
\emph{On the rate of convergence in homogenization of
Hamilton--Jacobi equations},
Indiana Univ. Math. J. \textbf{50} (2001), no.~3, 1113--1129.
\href{https://doi.org/10.1512/iumj.2001.50.1933}{doi:10.1512/iumj.2001.50.1933}.

\bibitem{CDq}
L.-P.~Chaintron and S.~Daudin,
\emph{Optimal rate of convergence in the vanishing viscosity
for quadratic Hamilton--Jacobi equations},
preprint, 2025.
\href{https://arxiv.org/abs/2502.09103}{arXiv:2502.09103}.

\bibitem{CD}
L.-P.~Chaintron and S.~Daudin,
\emph{Optimal rate of convergence in the vanishing viscosity
for uniformly convex Hamilton--Jacobi equations},
preprint, 2025.
\href{https://arxiv.org/abs/2506.13255}{arXiv:2506.13255}.

\bibitem{CG}
M.~Cirant and A.~Goffi,
\emph{Convergence rates for the vanishing viscosity approximation
of Hamilton--Jacobi equations: the convex case},
Indiana Univ. Math. J., to appear.
\href{https://arxiv.org/abs/2502.15495}{arXiv:2502.15495}.

\bibitem{CL83}
M.~G.~Crandall and P.-L.~Lions,
\emph{Viscosity solutions of Hamilton--Jacobi equations},
Trans. Amer. Math. Soc. \textbf{277} (1983), no.~1, 1--42.
\href{https://doi.org/10.1090/S0002-9947-1983-0690039-8}{doi:10.1090/S0002-9947-1983-0690039-8}.

\bibitem{CL84}
M.~G.~Crandall and P.-L.~Lions,
\emph{Two approximations of solutions of Hamilton--Jacobi equations},
Math. Comp. \textbf{43} (1984), no.~167, 1--19.
\href{https://www.ams.org/journals/mcom/1984-43-167/S0025-5718-1984-0744921-8/}{doi:10.1090/S0025-5718-1984-0744921-8}.

\bibitem{Evans89}
L.~C.~Evans,
\emph{The perturbed test function method for viscosity solutions
of nonlinear PDE},
Proc. Roy. Soc. Edinburgh Sect. A \textbf{111} (1989),
nos.~3--4, 359--375.
\href{https://doi.org/10.1017/S0308210500018631}{doi:10.1017/S0308210500018631}.

\bibitem{Evans92}
L.~C.~Evans,
\emph{Periodic homogenisation of certain fully nonlinear partial
differential equations},
Proc. Roy. Soc. Edinburgh Sect. A \textbf{120} (1992),
nos.~3--4, 245--265.
\href{https://doi.org/10.1017/S0308210500032121}{doi:10.1017/S0308210500032121}.

\bibitem{Evans10}
L.~C.~Evans,
\emph{Adjoint and compensated compactness methods for
Hamilton--Jacobi PDE},
Arch. Ration. Mech. Anal. \textbf{197} (2010), no.~3, 1053--1088.
\href{https://doi.org/10.1007/s00205-010-0307-9}{doi:10.1007/s00205-010-0307-9}.

\bibitem{Fleming64}
W.~H.~Fleming,
\emph{The convergence problem for differential games, II},
in \emph{Advances in Game Theory},
Ann. of Math. Stud. \textbf{52}, Princeton University Press,
1964, 195--210.
\href{https://doi.org/10.1515/9781400882014-013}{doi:10.1515/9781400882014-013}.

\bibitem{GJTZ}
X.~Guo, W.~Jing, H.~V.~Tran, and Y.~P.~Zhang,
\emph{Quantification of ergodicity for Hamilton--Jacobi equations
in a dynamic random environment},
preprint, 2026.
\href{https://arxiv.org/abs/2604.00315}{arXiv:2604.00315}.

\bibitem{HJ}
Y.~Han and J.~Jang,
\emph{Rate of convergence in periodic homogenization for convex
Hamilton--Jacobi equations with multiscales},
Nonlinearity \textbf{36} (2023), no.~10, 5279--5297.
\href{https://doi.org/10.1088/1361-6544/acf17c}{doi:10.1088/1361-6544/acf17c}.

\bibitem{HJMT}
Y.~Han, W.~Jing, H.~Mitake, and H.~V.~Tran,
\emph{Quantitative homogenization of state-constraint Hamilton--Jacobi
equations on perforated domains and applications},
Arch. Ration. Mech. Anal. \textbf{249} (2025), no.~2, Paper No.~18, 53~pp.
\href{https://doi.org/10.1007/s00205-025-02091-2}{doi:10.1007/s00205-025-02091-2}.

\bibitem{KL}
S.~Kim and K.-A.~Lee,
\emph{Higher order convergence rates in theory of homogenization III:
Viscous Hamilton--Jacobi equations},
J. Differential Equations \textbf{265} (2018), no.~10, 5384--5418.
\href{https://doi.org/10.1016/j.jde.2018.07.003}{doi:10.1016/j.jde.2018.07.003}.

\bibitem{LPV}
P.-L.~Lions, G.~Papanicolaou, and S.~R.~S.~Varadhan,
\emph{Homogenization of Hamilton--Jacobi equations},
unpublished manuscript, 1987.

\bibitem{LTY}
Z.~Liu, H.~V.~Tran, and Y.~Yu,
\emph{Sharp global and almost everywhere convergence rates for
periodic homogenization of viscous quadratic Hamilton--Jacobi equations},
preprint, 2026.
\href{https://arxiv.org/abs/2604.19948}{arXiv:2604.19948}.

\bibitem{MTY}
H.~Mitake, H.~V.~Tran, and Y.~Yu,
\emph{Rate of convergence in periodic homogenization of
Hamilton--Jacobi equations: the convex setting},
Arch. Ration. Mech. Anal. \textbf{233} (2019), no.~2, 901--934.
\href{https://doi.org/10.1007/s00205-019-01371-y}{doi:10.1007/s00205-019-01371-y}.

\bibitem{QSTY}
J.~Qian, T.~Sprekeler, H.~V.~Tran, and Y.~Yu,
\emph{Optimal rate of convergence in periodic homogenization of
viscous Hamilton--Jacobi equations},
Multiscale Model. Simul. \textbf{22} (2024), no.~4, 1558--1584.
\href{https://doi.org/10.1137/24M1642822}{doi:10.1137/24M1642822}.

\bibitem{Tran11}
H.~V.~Tran,
\emph{Adjoint methods for static Hamilton--Jacobi equations},
Calc. Var. Partial Differential Equations \textbf{41} (2011),
nos.~3--4, 301--319.
\href{https://doi.org/10.1007/s00526-010-0363-x}{doi:10.1007/s00526-010-0363-x}.

\bibitem{TranBook}
H.~V.~Tran,
\emph{Hamilton--Jacobi Equations: Theory and Applications},
Graduate Studies in Mathematics \textbf{213},
American Mathematical Society, Providence, RI, 2021.
\href{https://bookstore.ams.org/gsm-213/}{ISBN 978-1-4704-6511-7}.

\bibitem{TY}
H.~V.~Tran and Y.~Yu,
\emph{Optimal convergence rate for periodic homogenization of
convex Hamilton--Jacobi equations},
Indiana Univ. Math. J. \textbf{74} (2025), no.~3, 555--573.
\href{https://doi.org/10.1512/iumj.2025.74.60267}{doi:10.1512/iumj.2025.74.60267}.

\bibitem{Zeidler}
E.~Zeidler,
\emph{Applied Functional Analysis: Main Principles and Their Applications},
Applied Mathematical Sciences \textbf{109},
Springer, New York, 1995.
\href{https://doi.org/10.1007/978-1-4612-0821-1}{doi:10.1007/978-1-4612-0821-1}.

\end{thebibliography}
\end{document}